\documentclass[11pt,twoside]{amsart}
\usepackage{amsmath}
\usepackage{amsthm}
\usepackage{amsfonts}
\usepackage{amssymb}
\usepackage[all]{xy}
\usepackage{alltt}
\usepackage{enumitem}
\usepackage{stmaryrd}
\usepackage{comment}
\usepackage{xspace}

\theoremstyle{plain}
\newtheorem{prop}{Proposition}
\newtheorem{conjecture}[prop]{Conjecture}

\newtheorem{thm}[prop]{Theorem}
\theoremstyle{definition}

\newcommand{\Z}{\ensuremath{\mathbb{Z}}}

\newcommand{\Q}{\ensuremath{\mathbb{Q}}}

\DeclareMathOperator{\Hom}{Hom}
\DeclareMathOperator{\im}{im}
\DeclareMathOperator{\rank}{rank}
 
\newcommand{\n}{\noindent}

\usepackage[margin=3.15cm]{geometry}

\usepackage[
        colorlinks=true,urlcolor=blue,linkcolor=blue,citecolor=blue
]{hyperref}

\begin{document}  
 
\title{When is the multiplicative group of a field indecomposable?}
\date{\today}
 
\author{Sunil Chebolu}
\address{Department of Mathematics \\
Illinois State University \\
Normal, IL 61790, USA}
\email{schebol@ilstu.edu}

\author{Keir Lockridge} 
\address {Department of Mathematics \\
Gettysburg College \\
Gettysburg, PA 17325, USA}
\email{klockrid@gettysburg.edu}

\thanks{
The first author is supported by an NSA grant (H98230-13-1-0238) }

 
\begin{abstract}
The multiplicative group of a finite field is well known to be cyclic; in this note, we determine the finite fields whose multiplicative groups are direct sum indecomposable.  We obtain our classification using a direct argument and also as a corollary to Catalan's Conjecture. Turning to infinite fields, we prove that any infinite field whose characteristic is not equal to 2 must have a decomposable multiplicative group. We conjecture that this is also true for infinite fields of characteristic 2 and we narrow the class of possible counter-examples. Finally, using the classification of finite commutative primary rings with cyclic multiplicative groups, we determine all finite commutative rings with indecomposable multiplicative groups.
\end{abstract}
 
\maketitle
\thispagestyle{empty}



\section{Introduction}
Fermat primes and Mersenne primes are two central classes of prime numbers which have enjoyed great esteem in number theory. Spreading from the blackboards of professional mathematicians to the notebooks of amateurs, these primes and the various problems surrounding them have been a source of great inspiration and fascination.  As we investigate the question posed in the title of this paper, both classes will make a surprising entry onto the stage. The related problem of determining which abelian groups can occur as the multiplicative group of a field was raised by L\'{a}szl\'{o} Fuchs more than 50 years ago in \cite{Fuchs}. Since then, much progress has been made, but the problem remains unsolved (see, for example, \cite{May, Dicker, Eugene1, Hua}). We refer the reader to \cite{CMN} for a survey of results in this area. Fuchs also asked whether the torsion subgroup of the multiplicative group of a field is necessarily a summand; this question was answered negatively by Cohn in \cite{Cohn}. In this paper, the question we ask is very much in the spirit of the aforementioned work: which fields have indecomposable multiplicative groups?  (Recall that a group is said to be indecomposable if it cannot be written as direct sum of two non-trivial subgroups.) In \S \ref{ff}, we classify the finite fields with indecomposable multiplicative groups. Our argument is simple and direct. However, the main result may also be obtained as a corollary to Catalan's Conjecture, described in \S \ref{catalan}, which was proved in 2002 by the Swiss mathematician Preda Mih\u{a}ilescu. In \S \ref{if}, we consider infinite fields, where we show that any infinite field whose characteristic is not equal to 2 must have a decomposable multiplicative group. We are unable to resolve the characteristic 2 case, though we have narrowed the class of possible counter-examples. Finally, in \S \ref{fr}, we use the classification of finite commutative primary rings with cyclic unit groups (found in \cite{PS}) to determine all finite commutative rings with indecomposable unit groups. Throughout, we will use standard elementary facts from number theory, algebra, and group theory which may be found in \cite{Burton}, \cite{DummitFoote}, and \cite{Robinson}, respectively.

The structure of the group of units in a ring has been studied extensively, especially for finite rings and group rings. Examples where the unit group is saddled with a similarly strong simplifying condition include \cite{PS}, referred to above, and our recent work \cite{24, 12, CLY}, where we examined the conditions under which every non-trivial unit in a ring has order $p$ a prime.

\section{Finite Fields\label{ff}} 
The goal of this section is to prove Theorem \ref{finitefields}, classifying the finite fields with indecomposable multiplicative groups. Our proof in this section uses elementary methods; in \S \ref{catalan}, we obtain this classification as a corollary to Catalan's Conjecture.

We begin by recording two basic facts about finite fields which can be found in any standard algebra textbook; see \cite{DummitFoote}, for instance. First, recall that every field has prime power order, and for every prime $p$ and positive integer $r$, there is a unique (up to isomorphism) finite field whose order is $p^r$. (The prime $p$ is the characteristic of the field.) Second, recall that the multiplicative group of a finite field $F$, written $F^\times$, is cyclic. (In fact, the multiplicative group of any field is locally cyclic (\cite[1.3.4]{Roman}); i.e., every finite subgroup is a cyclic group.) In our first proposition we make a simple observation which follows from the structure theorem for finite abelian groups. We refer to a positive integer as a prime power if it is equal to $p^r$ for some prime $p$ and integer $r \geq 1$.

\begin{prop}
If $F$ is a finite field of order $p^{r}$, then $F^{\times}$ is indecomposable if and only if $p^{r}-1$ is either 1 or a prime power. 
\end{prop}
Note that $F^\times$ has order 1 if and only if $F = \mathbf{F}_2$, the finite field with two elements.
\begin{proof}
As mentioned above, the group $F^{\times}$ is isomorphic to $C_{p^{r}-1}$, the multiplicative cyclic group of order $p^r - 1$.   From the structure theorem for finite abelian groups, a finite cyclic group is indecomposable if and only if it is trivial or has prime power order. 
\end{proof}

Because there is a unique finite field corresponding to each prime power $p^{r}$, determining which finite fields have an indecomposable multiplicative group is equivalent to solving the following number theoretic problem: find all pairs $(p, r)$, where $p$ is prime and $r$ is a positive integer, such that $p^{r}-1$ is a  prime power. We begin by determining the pairs $(p, r)$ for which $p^{r}-1$ is $1$ or a power of $2$. Recall that a Fermat prime is a prime of the form $2^{2^{n}}+1$.

\begin{prop} \label{2powers}
The quantity $p^{r} - 1$ is a power of $2$ if and only if $p$ is a Fermat prime and $r = 1$ or $p = 3$ and $r = 2$.
\end{prop}
\begin{proof}   Suppose $p^{r} - 1$ is a power of $2$. Then $p^{r}- 1= 2^{n}$ for some positive integer $n$ and $p$ must be odd. Now,
$$(p -1)(p^{r-1} + \cdots + p + 1) = 2^{n},$$ 
and this implies that $p-1$ is a power of $2$, so $p$ is a Fermat prime. (If $2^{m}+1$ is prime, then it is well known that $m$ must be a power of 2.)   If $r = 1$, then we are done, but if $r \geq 2$, then the second factor is also divisible by 2, so $r$ must be even.  Thus, writing $r = 2v$, we have $$(p^v -1)(p^v + 1) = 2^{n}.$$  However, the only pair of positive integers that differ by two and are powers of two are 2 and 4, so it must be that $p^v = 3$. It follows that $(p, r) = (3, 2)$. 
\end{proof}

The next proposition gives the pairs $(p, r)$ for which $p^{r}-1$ is a  power of  an odd prime. Recall that a Mersenne prime is a prime of the form $2^r - 1$.

\begin{prop} \label{ppowers}
Suppose $p^r - 1 = q^n$ for some odd prime $q$ and positive integer $n$. Then, $n = 1$, $p = 2$, and $q = 2^r - 1$ is a Mersenne prime.
\end{prop}
\begin{proof}
First suppose $n = 1$ and $q = p^r -1$ is an odd prime. Parity considerations imply that $p = 2$, so $q = 2^r - 1$ is a Mersenne prime. 

We claim there are no other solutions to $p^r - 1 = q^n$. Indeed, assume to the contrary that $p^r - 1 = q^n$ with $n \geq 2$ and $q$ an odd prime. As above, we have $p = 2$ and $q^n + 1 = 2^r$. Since $q > 1$, we have $r \geq 2$, and hence $q^n \equiv -1 \, (4)$. This means $q \equiv -1\, (4)$ and $n$ is odd. We now have a factorization $$2^r = q^n + 1 = (q+1)(q^{n-1} - q^{n-2} + \cdots - q + 1).$$ Because $ n \geq 2$, the second factor is even. However, it is also congruent to the odd integer $n$ modulo 2, a contradiction. This proves that there are no solutions when $n \geq 2$, and the proof is complete.
\end{proof}

The following theorem now follows from the previous three propositions.

\begin{thm} \label{finitefields} Let $F$ be a finite field. The multiplicative group of $F$ is indecomposable if and only if $F$ is one of the following fields:
\begin{enumerate}
\item $\mathbf{F}_2$,
\item  $\mathbf{F}_9$,
\item $\mathbf{F}_p$ where $p$ is a Fermat prime, or
\item $\mathbf{F}_{q+1}$ where $q$ is a Mersenne prime.
\end{enumerate}
\end{thm}
\n We will see in the next section that the odd-ball case $\mathbf{F}_9$ corresponds to the unique solution in Catalan's Conjecture (see Theorem \ref{m-theorem}).

It is natural to ask whether there are infinitely many finite fields with an indecomposable multiplicative group.  In light of the above theorem, this question has an affirmative answer if and only if either the collection of Fermat primes or the collection of Mersenne primes is infinite. It is not known whether either collection is finite or infinite.  As of June 2014, only 5 Fermat primes and 48 Mersenne primes are known. It is believed that there are only finitely many Fermat primes and infinitely many Mersenne primes.

\section{Catalan's Conjecture\label{catalan}}
In 1844, the French mathematician Eug\`{e}ne Catalan conjectured that $8$ and $9$ are the only consecutive perfect powers among the positive integers. More precisely, he conjectured the following statement which was proved in 2002 by the Swiss mathematician Preda Mih\u{a}ilescu.

\begin{thm}[Mih\u{a}ilescu]\label{m-theorem}
The  Diophantine equation
\[ x^{u} - y^{v} = 1  \ \ \ \ (x \ge 1, y \ge 1, u \ge 2, v \ge 2)\]
has a unique solution which is given by $x^{u} = 3^{2}$ and $y^{v} = 2^{3}$.
\end{thm}
 
\noindent We refer the reader to \cite{Metsankyla} for the interesting history behind this theorem and an exposition of Mih\u{a}ilescu's proof.

We now make explicit the connection between Catalan's Conjecture and our eponymous problem. In \S \ref{ff}, we reduced the determination of the finite fields with (non-trivial) indecomposable multiplicative group to finding pairs $(p, r)$ (where $p$ is prime and $r$ is a positive integer) such that $p^{r}-1= q^{n}$ for some prime $q$ and positive integer $n$.  The last equation rearranges to $p^r - q^n = 1$. The connection to Catalan's Conjecture is now clear: we seek solutions to the Diophantine equation $x^{u} - y^{v} = 1$, where $x$ and $y$ are prime numbers and the exponents are natural numbers. Propositions \ref{2powers} and \ref{ppowers} may be together viewed as a special case of Catalan's Conjecture; they give a complete list of the consecutive prime powers.

We will now prove Theorem \ref{finitefields} using Catalan's Conjecture. Let $F$ be a finite field of order $p^{r}$ whose multiplicative group is indecomposable. Then, as explained above, we obtain $p^{r} - q^{n} = 1$, where $q$ is prime (here, we allow $n \geq 0$). We will consider 3 cases which neatly organize the fields obtained in Theorem \ref{finitefields}.

First, suppose $r = 1$ and $n \geq 0$.  This gives $p - q^{n} = 1$.  If $p = 2$, then $n = 0$ and $F = \mathbf{F}_2$.  If $p$ is odd, then $q = 2$ and $p$ is a Fermat prime. The corresponding finite fields are $\mathbf{F}_p$ where $p$ is a Fermat prime.

Next, suppose $r \geq 2$ and $0 \leq n \leq 1$. Here, $r \geq 2$ forces $n = 1$, so $p^{r} - q   = 1$.  If $q = 2$, then $p^{r} = 3$, which is not possible. If $q > 2$, then $p = 2$, and  $2^{r} - 1$ is a Mersenne prime. The corresponding finite fields are $\mathbf{F}_{q+1}$ where $q$ is a Mersenne prime. 

Finally, assume $r \ge 2$ and $n \ge 2$.  In this case, Catalan's Conjecture implies that $p^{2}= 3^{2}$. The corresponding finite field is $\mathbf{F}_{9}$.

\section{Infinite Fields\label{if}}

Our goal in this section is to determine all infinite fields with an indecomposable multiplicative group. We are currently unable to find a single example of such a field; we can, however, narrow the possible examples to a special class of fields of characteristic 2 (see Theorem \ref{reduction}).

To begin, let $F$ be an infinite field whose multiplicative group is indecomposable.  We will first argue that $F^\times$ must be torsion free. The following theorem relies upon a classical result of Pr\"{u}fer and Baer on the structure of abelian groups whose elements have boundedly finite orders (see \cite[\S 4]{Robinson} for details). Recall that a $p$-group is a group in which every non-trivial element has (finite) order a power of $p$.

\begin{thm}[{\cite[4.3.12]{Robinson}}]\label{corqc}
Let $G$ be an indecomposable abelian group that is not torsion-free. Then $G$ is either a cyclic or quasicyclic $p$-group for some prime $p$.
\end{thm}

A $p$-group is quasicyclic if it is isomorphic to $C_{p^\infty}$, the union of all cyclic groups of order a power of $p$:
\[ C_{p^{\infty}}  = \bigcup_{n\ge 0} C_{p^n}.\]
This group, also called the Pr\"{u}fer group, is written additively as the direct limit $\underset{\longrightarrow}{\lim}\, \mathbb{Z}/(p^n)$. It is the injective hull of $\mathbb{Z}/(p)$ and is isomorphic to $\mathbb{Z}\left[1/p\right]/\mathbb{Z}$.

Since our multiplicative group $F^\times$ is both infinite and indecomposable, Theorem \ref{corqc} implies $F^\times \cong C_{p^{\infty}}$.  We now prove that this is impossible.

\begin{prop}
There is no field $F$ whose multiplicative group is a quasicyclic $p$-group.
\end{prop}

\begin{proof}
Assume to the contrary that there is a field $F$ whose multiplicative group is isomorphic to $C_{p^{\infty}}$. Then every element in $F^\times$ is a torsion element. This implies that the characteristic of $F$ cannot be $0$. 
Let $q > 0$ be the characteristic of $F$.   Note that $F^\times$ contains a copy of $C_{p^i}$ for all positive integers $i$. In particular, $F$ contains $\zeta_{p^i}$, a primitive $p^i$th root of unity, for all $i \ge 1$. Consider the ascending tower of finite fields $\mathbf{F}_q( \zeta_{p^i})$ for $i \ge 1$ inside $F$. Let $q^{n_i}$ denote the orders of these finite fields. The multiplicative groups of these finite fields are finite cyclic subgroups of $C_{p^{\infty}}$. Consequently, there are integers $m_i$ such that
\label{keyequation}
\begin{equation} q^{n_i} - 1 = p^{m_i}\ \ \text{for}\ \   { i \ge 1}.\label{nozpi}\end{equation}
We now offer 3 different arguments to show that this is impossible. 

1. In Equation (\ref{nozpi}), note that $\{m_i\}$ and $\{n_i\}$ are both increasing sequences.  When $p=2$, this is impossible by Proposition \ref{2powers} and when $p$ is odd, this is impossible by Proposition \ref{ppowers}.

2. Note that Equation (\ref{nozpi}) can also be rewritten as
\[ q^{n_i} - p^{m_i}  = 1\ \ \text{for}\ \   { i \ge 1}.\]
This shows that there are infinitely many solutions in positive integers to the equation $x^{u} - y^{v} = 1$. This contradicts Catalan's Conjecture.

3. Our final argument relies on the following special case of Zsigmondy's Theorem, proved by A. S. Bang in 1886 (see \cite{Ribenboim} for the statement of Zsigmondy's Theorem and an account of its interesting history).
\begin{thm}[Bang] Let $a$ and $t$ be integers greater than $2$. There
exists a prime divisor $l$ of  $a^{t} - 1$ such that $l$ does not divide $a^{j} -1$ for all $
0 < j < t$.
\end{thm}
\n It is easy to see  that Equation (\ref{nozpi}) contradicts this theorem.\end{proof}

Now, coupling the above work with the observation that any field of characteristic not equal to 2 has non-trivial torsion element $-1$, we have the following theorem.

\begin{thm}\label{tfthm}
If $F$ is an infinite field such that $F^\times$ is not torsion-free, then $F^\times$ is a decomposable group. In particular, infinite fields of characteristic not equal to $2$ have decomposable multiplicative groups.  
\end{thm}

It is our suspicion that the characteristic 2 case is no different; we therefore make the following conjecture.
\begin{conjecture} Every infinite field has a decomposable multiplicative group. \end{conjecture}
\n Theorem \ref{reduction} summarizes everything we know about possible counter-examples to this conjecture. We will use the next proposition in that theorem. Recall that the rank of an abelian group $A$, written $\rank A$, is the size of a maximal linearly independent subset; equivalently, $\rank A$ is the dimension of $A\otimes \Q$ as a $\Q$-vector space. Further, since $\Q$ is flat over $\Z$, we have $\rank B \leq \rank A$ whenever $B$ is a subgroup of $A$.

\begin{prop}\label{free-summand}
Let $k$ be a field such that $k^\times$ is a free abelian group and let $K$ be a finite extension of $k$. Then, $K^\times$ has a summand isomorphic to $k^\times$.
\end{prop}
\begin{proof}
The field norm 
\[ N = N_{K/k} \colon K^\times \longrightarrow k^\times \]
is a non-trivial homomorphism of $\mathbb{Z}$-modules (see \cite[8.1.3]{Roman}). The restriction of $N$ to $k^\times$ is the $d$th-power map, where $d = [K\colon k]$.  The image of $N$ therefore contains $(k^\times)^d$, which is isomorphic to $k^\times$ since the latter group is a free $\Z$-module. We now have $$k^\times \cong (k^\times)^d \subseteq \im N \subseteq k^\times,$$ hence $\rank \im N = \rank k^\times$.  Since $\im N$ is a submodule of a free module, and submodules of free modules over a principal ideal domain are always free, we conclude that the image of the norm map is a free $\mathbb{Z}$-module whose rank is the same as the rank of $k^\times$; hence, $\im N \cong k^\times$. Further, since free $\mathbb{Z}$-modules are projective, the surjection
\[ N \colon K^\times \longrightarrow  \im N \]
splits. Thus $K^\times$ has a summand isomorphic to $k^\times$, as desired.
\end{proof}

\begin{thm}\label{reduction}
Let $F$ be an infinite field with indecomposable multiplicative group. Then, $F$ is an extension of $\mathbf{F}_2$, and there is an intermediate field $F \supseteq L \supsetneq \mathbf{F}_2$ such that $F$ is algebraic over $L$ and $L$ is a purely transcendental extension of $\mathbf{F}_2$.  The fields $F$ and $L$ must satisfy the following properties.
\begin{enumerate}
\item The group $L^\times$ is free abelian of infinite rank.\label{Lx}
\item $[F\colon L] = \infty$.\label{degree}
\item The field $F$ is a completely transcendental extension of $\mathbf{F}_2$. In particular, $F$ is not algebraically closed.\label{ct}
\item The group $F^\times$ is a torsion-free indecomposable abelian group of infinite rank. In particular, $F^\times$ is reduced and $\Hom(F^\times, \Z) = 0.$\label{fxprops}
\item The group $F^\times$ is an essential infinite union of subgroups, each having a free abelian group isomorphic to $L^\times$ as a summand.\label{essential}
\end{enumerate}
\end{thm}
Recall that an abelian group is reduced if it has no divisible subgroups. A completely transcendental extension $A/B$ is one where every element of $A\setminus B$ is transcendental over $B$. We say that a set $T$ is an essential union of a collection of subsets if each subset in the collection is necessary in covering the set $T$.
\begin{proof}
Since $F$ is infinite and $F^\times$ is indecomposable, we immediately obtain that $F^\times$ is torsion-free and $F$ is an extension of $\mathbf{F}_2$ by Theorem \ref{tfthm}. The existence of the intermediate field $L$ is a standard result in field theory. It is a proper extension of $\mathbf{F}_2$, for otherwise $F$ would contain elements algebraic over $\mathbf{F}_2$, contradicting (\ref{ct}), proved below.

Now consider (\ref{Lx}). Let $S$ be an algebraically independent set over $\mathbf{F}_2$ such that $L = \mathbf{F}_2(S)$. The ring $\mathbf{F}_{2}[S]$ of polynomials with indeterminates in $S$ and coefficients in $\mathbf{F}_2$ is a unique factorization domain, and $\mathbf{F}_2(S)$ is its field of fractions. We therefore have \[ L^\times \cong \bigoplus_{f \in \Delta} \mathbb{Z},\] where $\Delta$ is the set of irreducible polynomials in $\mathbf{F}_{2}[S]$. This is a free abelian group of infinite rank.

We next compute $[F\colon L]$. Assume to the contrary that $[F \colon L] < \infty $. We may then apply Proposition \ref{free-summand} to obtain that the free abelian group $L^\times$ of infinite rank is a summand of $F^\times$. This is impossible, however, as $F^\times$ is indecomposable. So $[F\colon L] = \infty$.

Now consider statement (\ref{ct}). No element in $F  \setminus \mathbf{F}_2$ can be algebraic over $\mathbf{F}_2$, for if $a$ in $F$ is algebraic over $\mathbf{F}_2$, then the subfield $\mathbf{F}_2(a)$, being a finite field different from $\mathbf{F}_2$, will contain non-trivial torsion elements, contradicting the fact that $F^\times$ is torsion-free. Thus $F$ is a completely transcendental extension of $\mathbf{F}_2$.  Since $F$ contains no roots of unity, it is certainly not algebraically closed.

For (\ref{fxprops}), we have already observed that $F^\times$ is torsion-free, and it is assumed indecomposable. It has infinite rank since it contains the subgroup $L^\times$ which has infinite rank.  To see that it is reduced, we first summon several facts from the homological algebra of abelian groups. See \cite[\S4]{Robinson} for details. The category of abelian groups is identical to the category of modules over the ring $\Z$. In this category, an abelian group is divisible if and only if it is an injective $\Z$-module. The salient property of injective modules here is that an injective module is a summand of any module in which it embeds. Now, suppose the indecomposable group $F^\times$ has a divisible subgroup. This subgroup must be a summand, and therefore $F^\times$ is itself divisible.  The only indecomposable divisible abelian groups are the quasicyclic groups $C_{p^\infty}$ and the rational numbers $\Q$. Since $F^\times$ is torsion-free, we must have $F^\times \cong \Q$. However, $\Q$ has rank 1, so this is not possible. Finally, if $\Hom(F^\times, \Z) \neq 0$, then $F^\times$ admits a nontrivial homomorphism onto an infinite cyclic group. Such groups are projective as $\Z$-modules, and any surjective map to a projective module splits. We therefore obtain that $\Z$ is a summand of $F^\times$, so $F^\times \cong \Z$. This is impossible, again because $F^\times$ has infinite rank.

Finally, we turn our attention to (\ref{essential}). Since $[F\colon L] = \infty$, we can express $F$ as an essential infinite union of finite extensions over $L$. The group $F^\times$ is an essential infinite union of the multiplicative groups of these finite extensions. Now apply Proposition \ref{free-summand} to each of these finite extensions to obtain that $L^\times$ is a summand of each of these subgroups.\end{proof}

\section{Finite Commutative Rings\label{fr}}
The problem under investigation can be easily generalized to rings.  Let $R$ be a commutative ring and let $R^\times$ denote the multiplicative group of units in $R$.  For which rings $R$ is $R^\times$ indecomposable? We provide an answer to this question for finite commutative rings in Theorem \ref{finiterings}.

Let $R$ be a finite commutative ring. The ring $R$ obviously satisfies the descending chain conditions on its ideals. That is, $R$ is an Artinian ring. It therefore decomposes as a direct product of finite commutative Artinian local rings  (see \cite[\S 8]{am}):
\begin{equation} R = R_1 \times R_2 \times \cdots \times R_k. \label{adecomp}\end{equation}
Since each $R_i$ is  Artinian and local, it has has a unique prime ideal (such rings are called primary rings). Thus, $k$ is the number of prime ideals in $R$.  Taking units, we obtain
\[ R^\times = R_1^\times \times R_2^\times \times \cdots \times R_k^\times.\]
The group $R^\times$ is indecomposable exactly when one factor is indecomposable and the remaining factors are trivial.  Since a non-trivial indecomposable finite abelian group is a cyclic group of prime power order, we have the following proposition.

\begin{prop}
Let $R$ be a finite commutative ring. The ring $R$ has an indecomposable multiplicative group of units if and only if the multiplicative group of exactly one factor in the decomposition (\ref{adecomp}) has prime power order and all the remaining factors have trivial multiplicative groups.
\end{prop}

This proposition reduces our problem to finding all  finite commutative Artinian local (primary) rings whose multiplicative group is either trivial or a cyclic group of prime power order. The more general problem of finding finite commutative primary rings with cyclic multiplicative groups was solved by Pearson and Schneider in \cite{PS}.

\begin{thm}[{\cite{PS}}]
Let $R$ be a finite commutative primary ring. The ring $R$ has a cyclic group of units if and only if $R$ is isomorphic to one of the following rings:
\begin{enumerate}
\item $\mathbf{F}_{q^{t}}$, where $q$ is prime and $t \ge 1$,
\item $\mathbb{Z}_{q^{s}}$, where $q$ is an odd prime and $s \ge 1$,
\item $\mathbb{Z}_{4}$,
\item $\mathbf{F}_{q}[x]/(x^{2})$, where $q$ is prime,
\item $\mathbf{F}_{2}[x]/(x^{3})$, or
\item $\mathbb{Z}_{4}[x]/(2x, x^{2}-2)$.
\end{enumerate}
\end{thm}

Computing the group of units for each of these rings is straightforward.  We summarize the results below.

\begin{enumerate}
\item $(\mathbf{F}_{q^{t}})^\times \cong C_{q^{t}-1}$. 
\item $(\mathbb{Z}_{q^{s}})^\times \cong C_{q^{s-1}(q-1)}$ (since $q$ is odd).  
\item $(\mathbb{Z}_{4})^\times \cong C_2$. 
\item $(\mathbf{F}_{q}[x]/(x^{2}) )^\times\cong C_{(q-1)q}$. 
\item $(\mathbf{F}_{2}[x]/(x^{3}))^\times \cong C_4$. 
\item $(\mathbb{Z}_{4}[x]/(2x, x^{2}-2))^\times \cong C_4$.
\end{enumerate}

We must now go through this list and isolate the rings whose multiplicative groups are either trivial or of prime power order. This, in conjunction with Theorem \ref{finitefields}, results in following theorem. (Note that the only finite primary ring with trivial multiplicative group of units is the ring $\mathbb{Z}_2$.) A ring is said to be indecomposable if cannot be expressed as a direct product of two non-zero rings. 

\begin{thm} \label{finiterings}
The following is a complete list of the finite commutative indecomposable rings which have an indecomposable multiplicative group:
\begin{enumerate}
\item $\mathbf{F}_2$,\label{first}
\item  $\mathbf{F}_9$,
\item $\mathbf{F}_p$, where $p$ is a Fermat prime,
\item $\mathbf{F}_{q+1}$, where $q$ is a Mersenne prime,
\item $\mathbb{Z}_{4}$,
\item $\mathbf{F}_{2}[x]/(x^{2})$,
\item $\mathbf{F}_{2}[x]/(x^{3})$, and
\item $\mathbb{Z}_{4}[x]/(2x, x^{2}-2)$.\label{last}
\end{enumerate}
A finite commutative ring $R$ has an indecomposable group of units if and only if $R$ is a (possibly empty) product of finitely many copies of $\Z_2$ and exactly one ring on the above list.
\end{thm}
\begin{proof}
It remains only to verify that each ring on the above list is indeed indecomposable; this is obvious for the first four rings, which are fields.  For the remaining rings, observe that a ring $R$ admits a non-trivial decomposition $R \cong R_1 \times R_2$ if and only if $R$ contains an idempotent element $e \neq 0, 1$. It is straightforward to check that the last four rings contain no such idempotent elements.\end{proof}

\bibliographystyle{alpha}

 

\end{document}